\documentclass[11pt, reqno]{amsart}
\usepackage[margin=3.6cm]{geometry}
\usepackage{array,booktabs,tabularx}
\usepackage{graphicx}
\usepackage{amssymb}
\usepackage{enumerate}
\usepackage{color}

\newtheorem{Thm}{Theorem}

\newtheorem{theorem}{Theorem}

\newtheorem{proposition}[theorem]{Proposition}
\newtheorem{lemma}[theorem]{Lemma}

\newtheorem{claim}[theorem]{Claim}

\theoremstyle{definition}
\newtheorem{definition}{Definition}
\newtheorem{remark}{Remark}

\newcommand{\leb}{\lambda}
\newcommand{\W}{\Omega}
\newcommand{\R}{{\mathbb R}}

\newcommand{\E}[1]{{\mathbb E}\left [#1\right]}

\newcommand{\ep}{\varepsilon}

\newcommand{\gives}{\ensuremath{\rightarrow}}

\newcommand{\abs}[1]{\ensuremath{\left| #1 \right|}}
\newcommand{\lr}[1]{\ensuremath{\left( #1 \right)}}
\newcommand{\norm}[1]{\left\lVert#1\right\rVert}

\newcommand{\w}{\omega}

\newcommand{\set}[1]{\ensuremath{\{#1\}}}

\DeclareMathOperator{\vol}{vol}

\DeclareMathOperator{\diam}{diam}
\DeclareMathOperator{\dist}{d}

\title[Eigenfunction Immersions]
{Fixed Frequency Eigenfunction Immersions and Supremum Norms of Random Waves}

\author[Y. Canzani]{Yaiza Canzani}
\author[B. Hanin]{Boris Hanin}

\address[Y. Canzani]{Department of Mathematics, Harvard University,  Cambridge, United States.\medskip}
 \email{canzani@math.harvard.edu}
\address[B. Hanin]{Department of Mathematics, Northwestern University,  Evanston, United States.\medskip}
\email{bhanin@math.northwestern.edu}
\begin{document}
\maketitle
%\tableofcontents

%----------------------------------------------------------------------
\begin{abstract}
A compact Riemannian manifold may be immersed into Euclidean space by using high frequency Laplace eigenfunctions. We study the geometry of the manifold viewed as a metric space endowed with the distance function from the ambient Euclidean space. As an application we give a new proof of a result of Burq-Lebeau and others on upper bounds for the sup-norms of random linear combinations of high frequency eigenfunctions.
\end{abstract}

%----------------------------------------------------------------------

\section{Introduction}
Let $(M,g)$ be an $n$-dimensional smooth compact Riemannian manifold without boundary and write $\Delta_g$ for the (non-negative) Laplace operator acting on $L^2(M,g)$.  For each $\lambda\geq 0$ we consider the space
\begin{equation}
\mathcal M_{\leb}=\bigoplus_{\mu\in (\leb,\leb+1]}\text{ker}(\Delta_g-\mu^2).\label{E:H Lambda}
\end{equation}
We set $m_\leb:= \text{dim}\,\mathcal M_\lambda$ and fix an orthonormal basis $\set{\varphi_{j,\leb}}_{j=1}^{m_\leb}$ of $\mathcal M_\leb$ consisting of eigenfunctions for $\Delta_g$ :
\begin{equation}\label{E:ONB} 
\Delta_g \varphi_{j,\leb}=\mu_{j,\leb}^2 \, \varphi_{j,\leb},\qquad \norm{\varphi_{j,\leb}}_{L^2}=1.
\end{equation} 
The purpose of this note is to prove a simple fact about the geometry of the immersions $\Phi_\leb:M\gives \R^{m_\leb}$ defined by 
\begin{equation}
\Phi_\leb(x):=\frac{1}{k_\leb}\lr{\varphi_{1,\leb}(x),\ldots, \varphi_{m_\leb,\leb}(x)} \qquad x \in M,\label{E:Phi Def}
\end{equation}
where
\begin{equation}\label{E:K Def}
k_{\leb}:=\sqrt{2}\;\frac{ \Gamma\lr{\frac{m_\leb+1}{2}}}{\Gamma\lr{\frac{m_\leb}{2}}}.
\end{equation}
That $\Phi_\leb$ are immersions for $\leb$ sufficiently large follows from Theorem \ref{T:AlmostIsom} below. The constants $k_\leb$ satisfy $k_\leb^2=  m_\leb +o(1)$ and, as we show in \S \ref{S:LWL}, ensure that
\begin{equation}
\lim_{\leb \gives \infty}\norm{\Phi_\leb(x)}_{l^2(\R^{m_\leb})}= \frac{1}{\sqrt{\vol_g(M)}}. \label{E:Diam Asymp}
\end{equation}
Each immersion $\Phi_\leb$ defines a (pseudo-)distance function $\dist_\leb:M \times M \to \R$ by restricting the ambient Euclidean $l^2$ distance:
\begin{equation}\label{E:dist_leb}
\dist_\leb(x,y):=\|\Phi_\leb(x) - \Phi_\leb(y)\|_{l^2(\R^{m_\leb})}.
\end{equation}

\noindent Our main result relates $\dist_\leb$ to the Riemannian distance $\dist_g$ on $M$ induced by $g.$ Recall that $(M,g)$ is said to be an \emph{aperiodic manifold} if for every $x\in M,$ the set of vectors $\xi$ in $T_xM$ for which the geodesic with initial conditions $\lr{x,\xi}$ returns to $x$ has Liouville measure $0.$ Any manifold with strictly negative sectional curvatures is aperiodic.  On the other extreme, Riemannian manifolds for which the geodesic flow is periodic are called \emph{Zoll manifolds}. Examples of Zoll manifolds are the spheres $S^n$ endowed with the round metric.

\begin{Thm}\label{T:Main}
Let $(M,g)$ be a smooth compact Riemannian manifold without boundary that is either Zoll or aperiodic. There exists a constant $C$ so that for all $x,y\in M$ and all $\leb\geq 0$
\begin{equation}
\dist_\leb(x,y)\leq C\cdot \leb \,\,\dist_g(x,y)\label{E:Pullback Dist Control}. 
\end{equation}
\end{Thm}

Relation \eqref{E:Diam Asymp} shows that the diameter of $M$ with respect to $\dist_\leb$ is bounded so that \eqref{E:Pullback Dist Control} is a non-trivial statement only when $\dist_g(x,y)$ is on the order of $\leb^{-1}.$ In this regime, much more can be said about $\dist_\leb(x,y)$. Indeed, on aperiodic manifolds the authors show in \cite{CH} that if $\leb\dist_\leb(x,y)$ is bounded as $\leb \gives \infty,$ then
\begin{equation}\label{E: precise distance asympt}
\dist_\leb(x,y)=\leb\cdot \frac{J_{\frac{n-2}{2}}(\leb \dist_g(x,y))}{\lr{\leb \dist_g(x,y)}^{\frac{n-2}{2}}}+o(\leb),
\end{equation}
where $J_\nu$ denotes the Bessel function of the first kind of order $\nu.$ The proof of this refined result is significantly more delicate than the proof of Theorem \ref{T:Main}.

Note that all the even spherical harmonics take the same values at antipodal points so that $d_\leb$ is only a pseudo-distance function in general. In contrast, the results in \cite{CH} prove that $d_\leb$ is an honest distance function if, for example, $(M,g)$ has negative sectional curvatures. We deduce Theorem \ref{T:Main} from the following estimate of Zelditch, which says that $\Phi_\leb$ is an almost-isometric immersion for $\leb$ sufficiently large. Let us write $g_{eucl}$ for the flat metric on $\R^d$ for any $d$ and introduce the pullback metrics
 \begin{equation}
g_\leb(x) := \Phi_\leb^*(g_{eucl})(x)=k_\leb^{-2}\sum_{j=1}^{m_\leb} d_x\phi_{j,\leb}(x)\otimes d_y\phi_{j,\leb}(y)\bigg|_{x=y} \label{E:Pullback Metric}
\end{equation}
for any $x\in M.$ 
\begin{Thm}[Zelditch \cite{Zel2}]\label{T:AlmostIsom}
  Let $(M,g)$ be a compact Riemannian manifold that is either Zoll or aperiodic with $\dim \, M =n$. Then, for any $x\in M,$
  \begin{equation}
g_\leb(x)=\frac{\alpha_n}{\lr{2\pi}^n} \cdot \,\leb^2 \,g(x)\lr{1+o(1)}\label{E:AlmostIsom}
%g_\leb(x)=\frac{\alpha_n}{\lr{2\pi}^n} \,\frac{\leb^{n+1}}{k_\leb^2} \,g(x)\lr{1+o(1)}\label{E:AlmostIsom}
\end{equation}
as $\leb\gives \infty,$ where $\alpha_n$ denotes the volume of the unit ball in $\R^n.$
\end{Thm}

Equation \eqref{E:AlmostIsom} allows us to relate $\dist_g$ to the distance function $\dist_{g_\leb}$ of $g_\leb.$ Theorem \ref{T:Main} then follows by observing that $\dist_\leb\leq \dist_{g_\leb}$ (see \S \ref{S:Proof} for details). 

The maps $\Phi_\leb$ are the Riemannian analogs of Kodaira-type projective embeddings $\Psi_N:M\gives \mathbb P H^0(L^{\otimes N})^{\lor}$ of a compact K\"ahler manifold $M$ into the projectivization of the dual of the space of global sections $H^0(L^{\otimes N})^{\lor}$ of the $N^{th}$ tensor power of an ample holomorphic line bundle $L \twoheadrightarrow M.$ Just as the choice of a Riemannian metric $g$ gives a weighted $L^2$ space on $M,$ a Hermitian metric $h$ on $L$ induces a weighted $L^2$ inner product on $H^0(L^{\otimes N}).$ This inner product gives rise to a Fubini Study metric $\w_{FS}^N$ on $\mathbb P H^0(L^{\otimes N})^{\lor},$ which plays the role of the Euclidean flat metric $g_{eucl}$ used in the present article. 

The holmomorphic analog of Theorem \ref{T:AlmostIsom} is that the pullback metrics $\Psi_N^*\lr{\w_{FS}^N}$ converge in the $C^\infty$ topology to the curvature form of $h.$ This statement is a result of Zelditch in \cite{Zel4}, with weaker versions going back to Tian \cite{Tia}. Theorem \ref{T:Main} in this holomorphic context follows easily from the holomorophic result of Zelditch and was used to study the sup norms of random homomorphic sections by Feng and Zelditch in \cite{FeZ}.

\subsection{Application to Sup Norms of Random Waves} Our main application of Theorem \ref{T:Main} is to find upper bounds on $L^\infty$-norms of random waves on $(M,g).$ \begin{definition}\label{D:Gaussian Ensemble}
A \textit{Gaussian random wave of frequency $\lambda$} on $(M,g)$ is a random function $\phi_{\leb}\in \mathcal M_\leb$ defined by
\begin{equation*}\label{eq: f}
\phi_{\lambda}:=\sum_{j=1}^{m_\leb} a_{j,\leb} \,\varphi_{j,\leb},
\end{equation*}
where the $a_{j,\leb}\sim N(0,k_{\leb}^{-2})$ are independent and identically distributed standard real Gaussian random variables and $\set{\varphi_{j,\leb}}_j$ is an orthonormal basis for $\mathcal M_\leb$ consisting of Laplace eigenfunctions as defined in \eqref{E:ONB}. 
\end{definition}
The choice of $k_\leb$ makes $\mathbb E\left[\norm{\phi_\leb}_2\right]=1.$ Gaussian random waves were introduced by Zelditch in \cite{Zel2}, and, in addition to their $L^p-$norms, a number of subsequent articles have studied their zero sets and critical points (cf eg \cite{KPW,FZW,Nic} and references thererin). 

The statistical features of Gaussian random waves of frequency $\leb$ are uniquely determined by their so-called canonical distance, which is precisely $d_\leb$ (cf Definition \ref{D:Canonical Metric}). In particular, upper bounds on the expected value of their $L^\infty$-norms are related to the metric entropy of their canonical distance via Dudley's entropy method (see \S \ref{S:Dudley}). We use Theorem \ref{T:Main} to prove in \S \ref{S:Eval Proof} the following result. 
\begin{Thm}\label{T:Eval}
Let $(M,g)$ be a smooth compact boundaryless Riemannian manifold of dimension $n$. Assume that  $(M,g)$ is either aperiodic or Zoll, and let $\phi_{\leb}$ be a random wave of frequency $\leb.$ Then 
\begin{equation}
\limsup_{\leb \to +\infty} \frac{\mathbb E\left[ \norm{\phi_\leb}_\infty\right]}{\sqrt{\log \leb}} \leq 16\sqrt{\frac{2n}{\vol_g(M)}}.\label{E:Eval Est}
\end{equation}
\end{Thm}

\begin{remark}\label{R: Eval aperiodic}
If $(M,g)$ is aperiodic, the more precise control of the distance function described in \eqref{E: precise distance asympt} yields a slight improvement of Theorem \ref{T:Eval} by reducing the constant on the right hand side to $16\sqrt{\frac{n}{\vol_g(M)}}.$
We shall indicate how to use \eqref{E: precise distance asympt} to get the improved upper bound in Remark \ref{R: proof of Eval aperiodic}.
\end{remark}
The upper bounds  of order $\sqrt{\log \leb}$ are not new. Indeed, a simple computation shows that $\mathbb E \left[\norm{\phi_\leb}_\infty\right]= \mathbb E\left[\norm{\psi_\leb}_\infty\right]$ if $\psi_\leb$ is chosen uniformly at random from the unit sphere
$$S\mathcal M_\leb:=\{f \in \mathcal M_\leb:\;\|f\|_2=1\} $$ 
endowed with the uniform probability measure. For the $L^2$-normalized random waves $\psi_\leb,$ the expectation of the $L^\infty$-norms has been studied on many occasions. On round spheres $(S^n, g_{round}),$ VanderKam obtained in \cite{Van} that $\mathbb E\left[ \norm{\phi_\leb}_\infty\right] =O(\log^2\leb)$.  Later, Neuheisel in \cite{Neu} improved the bound to $\mathbb E\left[ \norm{\phi_\leb}_\infty\right] =O(\sqrt{\log \leb})$. On general smooth, compact, boundaryless Riemannian manifolds Burq and Lebeau \cite{BL} proved the existence of two positive constants $C_1,C_2$ so that as $\leb \to \infty$
$$C_1 \sqrt{\log \leb} \leq \mathbb E\left[ \norm{\phi_\leb}_\infty\right] \leq C_2\sqrt{\log \leb}.$$
Although an explicit value for $C_2$ is not stated in \cite{BL}, Burq relayed to the authors in a private communication how one may be extracted. Our constant $16\sqrt{2}$ multiplying $\lr{n/{\vol_g(M)}}^{\frac{1}{2}}$ is larger (i.e. worse) than theirs, which is approximately $e^{-e}+1+\frac{1}{\sqrt{2e}}.$ The reason for the discrepancy is that Dudley's entropy method makes no assumption on the structure of the probability space on which the Gaussian field in question is defined, while Burq and Lebeau use explicit concentration results on the spheres $S\mathcal M_\leb.$ There is a partial converse to Dudley's entropy method, called Sudakov minoration, which allows one to obtain lower bounds on the sup norms of Gaussian fields. Even with the refined control \eqref{E: precise distance asympt} on the canonical distance $\dist_\leb$ from \cite{CH}, the general lower bounds seem to be on the order of $\leb^{-1}\sqrt{\log \leb},$ which are significantly worse than those given by Burq and Lebeau. Finally, we mention that \cite{BL} studies more generally upper bounds on $L^p$ norms, $p\in [2,\infty),$ of random waves. Results on $L^p$ norms are also contained in the work of Ayache-Tzvetkov \cite{ATz} and Tzvetkov \cite{Tz}.

%----------------------------------------------------------------------

\section{Acknowledgements} We would like to thank J. Toth and S. Zelditch for many useful comments on an earlier draft of this article. We would also like to thank N. Burq and S. Eswarathasan for pointing us to the work of N. Burq and G. Lebeau \cite{BL}, which gives the matching upper and lower bounds for the sup-norms of random waves treated in Theorem \ref{T:Eval}. We would particularly like to thank N. Burq for explaining how to extract an explicit constant from the proof in \cite{BL}. The second author would also like to thank D. Baskin for several helpful conversations regarding the relationship between $d_\leb$ and $d_{g_\leb}.$

\section{Preliminaries }
The proof of Theorem \ref{T:Main} relies on three ingredients: the local Weyl law (\S \ref{S:LWL}), Dudley's entropy method (\S \ref{S:Dudley}), and Zelditch's almost-isometry result (Theorem \ref{T:AlmostIsom} in the Introduction). Throughout this section $(M,g)$ denotes a compact Riemannian manifold of dimension $n$.

\subsection{Asymptotics for the Spectral Projector}\label{S:LWL}
For $x,y\in M$, we write
\[E_{_{(0,\leb]}}(x,y)=\sum_{\leb_j\in (0,\leb]}\varphi_j(x)\varphi_j(y)\]
for the Schwartz kernel of the orthogonal projection 
\[E_{(0,\leb]}:L^2(M,g)\twoheadrightarrow \bigoplus_{\mu\in [0,\leb)} \ker\lr{\Delta_g -\mu^2}.\] 
For $\lambda>0$ we define 
\[N(\lambda):=\dim \left[\text{range} \lr{E_{(0,\leb]}}\right]\] 
to be the number of eigenvalues of $\Delta_g$ smaller than $\lambda^2$ counted with multiplicity and set
 $$\alpha_n:= \frac{\omega_n}{(2\pi)^n},$$ 
where $\omega_n$ is the volume of the unit ball in $\R^n$. We also write 
\[E_{_{(\leb,\leb+1]}}(x,y)=\sum_{\leb_j\in (\leb,\leb+1]}\varphi_j(x)\varphi_j(y)\]
for the kernel of the orthogonal projection onto the span of the eigenfunctions of $\Delta_g$ whose eigenvalues lie in $(\leb^2, \lr{\leb+1}^2]$. On several occasions we use the following result.
\begin{proposition}[Local Weyl Law \cite{DG} and \cite{Ivr}]\label{P:LWL}
Let $(M,g)$ be a compact Riemannian manifold of dimension $n$ that is either Zoll or aperiodic. Then,
  \begin{align}
&E_{_{(0,\leb]}}(x,x)= \alpha_n \leb^n +o(\leb^{n-1}) && \text{as}\;\; \lambda \to \infty \notag \intertext{and thus}
&E_{_{(\lambda,\lambda+1]}}(x,x)= n \alpha_n \leb^{n-1}+o(\leb^{n-1}) &&\text{as}\;\; \lambda \to \infty, \label{E:LWL}
  \end{align}
with the implied constants being uniform in $\leb$ and $x\in M.$
\end{proposition}
\noindent Integrating the above expressions, one has that on compact aperiodic (or Zoll) manifolds
$$N(\lambda)= \alpha_n\text{vol}_g(M) \lambda^n + o(\lambda^{n-1}) \qquad  \;\text{as}\;\; \lambda \to \infty.$$
Continuing to write $m_\leb=\dim \mathcal M_\leb,$ we see that
\begin{equation}\label{d_lambda}
m_\lambda= N(\leb+1)-N(\leb) =n \alpha_n \text{vol}_g(M) \lambda^{n-1} + o(\lambda^{n-1})\qquad \text{as}\;\; \lambda \to \infty.
\end{equation}

\noindent We also need H\"ormander's off-diagonal pointwise Weyl law: 
\begin{lemma}[\cite{Hor}]\label{L:Off-Diag LWL}
  Let $(M,g)$ be a compact Riemannian manifold of dimension $n.$ Fix $x,y\in M.$ Then
  \begin{equation}
    \label{E:Off-Diag LWL}
    E_{_{(0,\leb]}}(x,y)=O(\leb^{n-1}).
  \end{equation}
\end{lemma}

\noindent 
Using the definition \eqref{E:K Def} of $k_{\leb},$ equation \eqref{d_lambda}, and the fact that $\Gamma\lr{x+\frac{1}{2}}/\Gamma\lr{x}= \sqrt{x}+O(1/\sqrt{x})$ as $x\gives \infty,$ we have from \eqref{d_lambda} that
\begin{align} \label{E:k Asymp}
k_{\leb}^2
&=m_\leb+O(\leb^{-\frac{n-1}{2}}) \notag \\ 
&= n \,\alpha_n \text{vol}_g(M) \lambda^{n-1} + o(\lambda^{n-1}).
\end{align} 
Finally, by combining   \eqref{E:LWL} and \eqref{E:k Asymp}, we see that for all $x\in M$
\begin{equation}
\norm{\Phi_\leb(x)}_{l^2(\R^{m_\leb})}=\frac{1}{k_\leb}\sqrt{E_{_{(\leb,\leb+1]}}(x,x)}=\frac{1}{\sqrt{\vol_g(M)}}+o(1)\label{E:Asymp Diam}
\end{equation}
as $\leb \gives \infty,$ which confirms \eqref{E:Diam Asymp}.

\subsection{Dudley's Entropy Method}\label{S:Dudley}

 Let $(\Omega, \mathcal A, \mathbb P)$ be a complete probability space. A measurable mapping $\phi:\Omega \to \R^M$ is called a \emph{random field} on $M$. If for every finite collection $\set{x_j}_{j=1}^N\in M$ the random vector $\set{\phi(x_j)}_{j=1}^N$ is Gaussian, then $\phi$ is said to be a Gaussian field on $M$.  In addition, if $\E{\phi(x)}=0$ for all $x \in M$, then $\phi$ is said to be centered. The Gaussian random waves of Definition \ref{D:Gaussian Ensemble} on $(M,g)$ are examples of centered Gaussian random fields on $M$.

\begin{definition}\label{D:Canonical Metric}
Let $\phi$ be a centered Gaussian field on $M.$ The \textit{canonical distance} on $M$ induced by $\phi$ is
\[\dist_\phi(x,y):=\left(\E{(\phi(x)-\phi(y))^2}\right)^{\frac{1}{2}} \qquad \text{for}\;\;x,y \in M.\]
\end{definition}

The law of of any centered Gaussian field is determined completely by its canonical distance function. We note that in general $\dist_\phi$ turns $M$ into a pseudo-metric space. Let us define for each $\ep>0$ the $\ep$-covering number of $(M,\dist_{\phi})$ as
\begin{equation}
N_{\dist_\phi}(\ep):=\inf \Big \{\ell \geq 1 \,: \, \exists\, x_1,\ldots, x_\ell \in M \text{~such that~} \cup_{j=1}^\ell B_{\ep}(x_j)=M \Big\},\label{E:Met Ent}
\end{equation}
where $B_{\ep}(x)$ is the ball of radius $\ep$ centered at $x.$ Dudley's entropy method says that 
\begin{equation}
\E{ \sup_{x\in M} \phi(x)} \leq 8\sqrt{2} \int_{0}^{D_\phi} \sqrt{ \log N_{\dist_\phi}(\ep) }\; d\ep, \label{E:Dudley}
\end{equation}
where $D_\phi=\text{diam}(M,\dist_\phi)/2$ (see for example \cite[Theorem 1.3.3]{AT} for the statement with an unspecified constant and the notes \cite{Bar} of Bartlett where the constant $8\sqrt{2}$ appears).

We observe that if $(M,g)$ is a smooth compact Riemannian manifold and $\phi_\leb$ is a random wave on $M$ with frequency $\leb,$ then it follows from Definition \ref{D:Gaussian Ensemble} that $\dist_\leb$ and $\dist_{\phi_\leb}$ coincide. Indeed,  for $x,y \in M$
\begin{align}
\dist_{\leb}^2(x,y)
&=\|\Phi_\leb(x)-\Phi_\leb(y)\|^2 \notag\\
&=\frac{1}{k_\leb^2} \Big(E_{_{(\leb,\leb+1]}}(x,x)+E_{_{(\leb,\leb+1]}}(y,y)-2E_{_{(\leb,\leb+1]}}(x,y) \Big) \label{E:dist and projectors} \\
&=\E{(\phi_\leb(x)-\phi_\leb(y))^2} \label{E:Can Dist}.
\end{align}

\section{Proof of Theorem \ref{T:Main}}\label{S:Proof}
The starting point for our proof is the following estimate on $\diam(M, \dist_\leb),$ the diameter of $M$ with respect to $\dist_\leb.$
\begin{proposition}\label{diameter}Let $(M,g)$ be a smooth compact boundaryless Riemannian manifold of dimension $n$. Assume that $(M,g)$ is either aperiodic or Zoll. There exists $\delta>0$ so that as $\leb\gives \infty$
  \begin{equation}
\delta \leq \diam(M,\dist_\leb)\leq \frac{2}{\sqrt{\vol_g(M)}}+o(1) .\label{E:Diam Estimate}
\end{equation}
\end{proposition}

\begin{remark}\label{R: diameter aperiodic} The authors' results in a forthcoming paper \cite{CH} allow one to prove matching upper and lower bounds in \eqref{E:Diam Estimate} and give the precise asymptotic: 
$ \diam(M, \dist_\leb)= \frac{\sqrt{2}}{\sqrt{\vol_g(M)}}+o(1).$ The corresponding statement for Zoll manifolds is simpler and follows from the fact that $E_{(0,\leb]}$ has a complete asymptotic expansion (cf \cite{Zel3}).\end{remark}

\begin{proof}
For $x,y \in M$ we have
\begin{align}
\dist_{\lambda}^2(x,y)=\frac{\norm{\Phi_{\leb}(x)-\Phi_{\leb}(y)}^2}{k_{\leb}^2} \label{E:Dist Emb}.
\end{align}
From \eqref{E:k Asymp} and the fact that $\norm{\Phi_{\leb}(x)}^2=\Pi_{_{(\lambda,\lambda+1]}}(x,x)$ for all $x\in M$ we conclude
\[\dist_{\lambda}^2(x,y)=\frac{\norm{\Phi_{\leb}(x)-\Phi_{\leb}(y)}^2}{m_{\lambda}+O(1)}\leq \frac{4}{\vol_g(M)}+o(1).\]
Taking the supremum over $x,y\in M$ proves the upper bound in \eqref{E:Diam Estimate}. To prove the lower bound in \eqref{E:Diam Estimate} we proceed by contradiction. That is, suppose that $\dist_{\leb}^2(x,y)$ is not bounded below. In virtue of (\ref{E:Dist Emb}) and (\ref{E:k Asymp}), this means that
\begin{equation}
\norm{\Phi_{\leb}(x)-\Phi_{\leb}(y)}=o(\leb^{n-1})\label{E:Assumption}
\end{equation}
for all $x,y\in M.$ Consider the map $\Psi_{\leb}:M\gives \R^{N(\leb)}$ given by 
\[\Psi_{\leb}(x)=\lr{\varphi_1(x),\ldots, \varphi_{{N(\leb)}}(x)} \qquad \text{for } x\in M,\]
where we continue to write $N(\leb)$ for the number of eigenvalues of $\Delta_g$ in the interval $(0,\leb^2].$ Note that the difference between $\Psi_{\leb}$ and $\Phi_{\leb}$ is that $\Psi_\leb$ includes all the eigenfunctions up to eigenvalue $\leb^2.$ The local Weyl law \eqref{E:Off-Diag LWL} shows that for $x,y \in M$ with $x \neq y$, 
\begin{equation}
\norm{\Psi_{\leb}(x)-\Psi_{\leb}(y)}^2=E_{_{(0,\leb]}}(x,x)+E_{_{(0,\leb]}}(y,y)-2 E_{_{(0,\leb]}}(x,y)=\W\lr{\leb^n}.\label{E:Big Diam Growth}
\end{equation}
On the other hand, for any positive integer $k$
\[\norm{\Psi_k(x)-\Psi_k(y)}^2=\sum_{j=0}^{k-1}\norm{\Phi_j(x)-\Phi_j(y)}^2.\]
The assumption \eqref{E:Assumption} shows that
\[\norm{\Psi_k(x)-\Psi_k(y)}^2=o(k^n),\]
which contradicts \eqref{E:Big Diam Growth}.
\end{proof}

Let us denote by $\dist_{g_{\leb}}$ the distance function for the metric $g_{\leb}$ defined in \eqref{E:Pullback Metric}. We continue to write $\dist_g$ for the distance function of the background metric $g.$ Combining \eqref{E:Diam Estimate} with \eqref{E:k Asymp}, we see that Theorem \ref{T:Main} reduces to showing that for any $K>2/\sqrt{\vol_g(M)}$ there exist positive constants $C, \lambda_0$ depending on $(M,g)$ such that
$\frac{\dist_\leb(x,y)}{\dist_g(x,y)} \leq C  \lambda$
for all $\lambda \geq \lambda_0$ and  $x,y \in M$ with $\dist_g(x,y) \leq K /\lambda.$ We write
\begin{equation}\label{quotients}
\frac{\dist_\leb(x,y)}{\dist_g(x,y)}=\frac{\dist_\leb(x,y)}{\dist_{g_{\leb}}(x,y)}\cdot \frac{\dist_{g_{\leb}}(x,y)}{\dist_g(x,y)}.
\end{equation}

Note that $\frac{\dist_\leb(x,y)}{\dist_{g_{\leb}}(x,y)}\leq 1.$ Indeed, let $\gamma:[0,1]\gives M$ be any length minimizing geodesic between two given points $x$ and $y$ with respect to the metric $g_\leb=\Phi_\leb^*(g_{eucl})$ and consider the curve $ \Phi \circ \gamma:[0,1] \to \R^{m_\leb}$ joining $\Phi_\leb(x)$ and $\Phi_\leb(y)$. Then 
  \begin{align} \label{quotient 1}
  \dist_\leb(x,y)&= \|\Phi_\leb(x)-\Phi_\leb(y)\|_{g_{eucl}}  \leq \int_0^1 \norm{\tfrac{d}{dt}\Phi_\leb (\gamma(t))}_{g_{eucl}} dt  = \dist_{g_\leb}(x,y).
  \end{align}
We turn to show that there exists $C>0$ such that $\frac{\dist_{g_{\leb}}(x,y)}{\dist_g(x,y)} \leq C\leb$.
Fix $x,y \in M$ and a unit speed geodesic $\gamma$ from $x$ to $y$ with respect to the metric $g.$ 
Theorem \ref{T:AlmostIsom} states that there exists $C>0$ so that as $\leb \to \infty$ one has
$g_\leb=C\cdot \frac{\leb^{n+1}}{k_\leb} g\lr{1+o(1)}$.
Therefore, after suitably  adjusting $C$,
\begin{equation}\label{quotient 2}
  \dist_{g_\leb}(x,y)  \leq \int_0^{\dist_g(x,y)}\norm{\tfrac{d}{dt}\gamma(t)}_{g_\leb} dt \leq C\, \sqrt{\frac{ \leb^{n+1}}{k_\leb} }\dist_g(x,y).
\end{equation}
\noindent Substituting \eqref{quotient 1} and \eqref{quotient 2} into \eqref{quotients}, and using the asymptotics \eqref{E:k Asymp} for $k_\leb$ completes the proof of Theorem \ref{T:Main}.

%----------------------------------------------------------------------

\section{Proof of Theorem \ref{T:Eval}}\label{S:Eval Proof}
\noindent Let $\phi_\leb$ be a Gaussian random wave of $(M,g)$ with frequency $\leb$. As explained in \eqref{E:Can Dist} the distance $\dist_\leb$ induced by the immersions and the distance $\dist_{\phi_\leb}$ induced by $\phi_\leb$ coincide. Theorem \ref{T:Main} therefore allows us to control the metric entropy $N_{\dist_\leb}$ of $\dist_{\phi_\leb}$ (see \eqref{E:Met Ent}).

\begin{proposition}\label{covering number}
Let $(M,g)$ be a compact aperiodic or Zoll Riemannian manifold. There exist positive constants $C$ and $\leb_0$ such that if
$\phi_\lambda$ is a Gaussian random wave of  frequency $\lambda$, then
\[N_{\dist_{\lambda}}(\ep)\leq \, C \,\,\frac {\lambda^n}{\ep^n}\]
for all $\ep>0$  and $\leb \geq \leb_0$.
\end{proposition}
\begin{proof}
From Theorem \ref{T:Main} it follows that if $x,y \in M$ are such that $\dist_g(x,y)\leq \frac{\ep^2}{C\lambda}$ and $\leb$ exceeds some fixed $\leb_0$, then $\dist_{\leb}(x,y)\leq  \,{\ep}$. Writing $B_d(x,r)$ for a ball of radius $r$ in the distance $d,$ we see that $B_{\dist_g} \lr{x,\frac{\ep^2}{C\lambda}}\subseteq B_{\dist_{\lambda}}(x,\ep)$ for each $x\in M.$ Hence,
\begin{equation}\label{E: covering 1}
N_{\dist_{{_\lambda}}}(\ep)\leq N_{\dist_g}\lr{\frac{\ep^2}{C\lambda}}
\end{equation}
for every $\leb\geq \leb_0.$ Define $\alpha_g:=\inf\{ k_g(x):\; x\in M\}$ where $k_g$ denotes the sectional curvature function for $(M,g)$,
and  set $\pi/\sqrt{\alpha_g}=\infty$ whenever $\alpha_g \leq 0$. For $\lambda$ large enough
$$\frac{\ep^2}{C\lambda}<\min \Big\{ \text{inj}_g(M),  \frac{\pi}{\sqrt{\alpha_g}}, 2\pi \Big\},$$
and so we may apply \cite[Lemma 4.1]{LP} to obtain
\begin{equation}\label{E: covering 2}
N_{\dist_g}\lr{\frac{\ep^2}{C\lambda}} \leq  \text{vol}_g(M)  \frac{2n}{s_{n-1}}  \pi^{n-1}\, \lr{\frac{\ep^2}{C\lambda}}^{-n},
\end{equation}
where $s_{n-1}$ is the volume of the $(n-1)$-dimensional unit sphere in $\R^n$.
The claim follows from combining \eqref{E: covering 1} and \eqref{E: covering 2}.
\end{proof}

To complete the proof of Theorem \ref{T:Eval}, we input the upper bound on $N_{\dist_{\leb}}$ of Proposition \ref{covering number} into Dudley's entropy estimate \eqref{E:Dudley} to conclude that for a constant $C$ depending $(M,g)$ and all $\leb$ exceeding some $\leb_0$,
\begin{align}
\label{E:Dudley Estimate} \E{\sup_{x \in M} \phi_{_\lambda}(x)}
&\leq 8\sqrt{2n} \int_{0}^{D_\lambda} \sqrt{ \log \Big (\frac{ C^{{1/}{n}}\lambda }{\ep} }\Big)\; d\ep \notag\\
&=8 D_\lambda \sqrt{2n} \int_{0}^{1} \sqrt{ \log \Big(\frac{\alpha_\lambda }{ \ep} }\Big)\; d\ep \notag,
\end{align}
where $D_\lambda=\diam(M,d_\leb)/2$ and 
\begin{equation}
\alpha_\leb:= \frac{C^{1/n}}{D_\lambda} \lambda.\label{E:Step 0}
\end{equation}
Setting $a_\lambda:= 1/ \log \alpha_\leb$ we get
\begin{equation}
\E{\sup_{x \in M} \phi_{_\lambda}(x)}\leq 8\sqrt{2} D_\leb \sqrt{n}\, \sqrt{\log \alpha_\leb} \int_0^1 (1-a_\leb \log \ep)^{1/2}\; d\ep.\label{E:Step 1}
\end{equation}
We now use the estimate 
\begin{equation}
\label{E:claim}\abs{\int_0^1 \lr{1-a_\leb\cdot \log \ep}^{1/2}d\ep-1}\leq \frac{a_{\leb}}{2}  \qquad \text{as~}a_\leb \to 0,
\end{equation}
whose proof we give in Claim \ref{Claim} below. Hence, combining \eqref{E:claim} with \eqref{E:Step 1} and the definition \eqref{E:Step 0} of $\alpha_{\leb}$, we have
\begin{equation}\label{sup bound}
\E{ \sup_{x \in M} \phi_{\lambda}(x)} \leq 8\sqrt{2} D_\lambda \sqrt{n}\sqrt{\log \leb + \log \lr{C^{1/n}/D_\leb}}\lr{1+\frac{a_{\leb}}{2}}.
\end{equation}
Proposition \ref{diameter} guarantees that $ D_\leb \leq  \frac{1}{\sqrt{\vol_g(M)}}+o(1)$. Finally, fix $\ep>0$. We claim that we may choose $\leb_0$ so that for $\leb \geq \leb_0,$
 \begin{equation}\label{E: eqn 1}
 \frac{1+\log \lr{C^{1/n}/D_\leb}}{\log \leb}\leq \lr{1+\frac{\ep}{16\sqrt{2}}}^{2/3}.
 \end{equation}
Indeed, this could be done if we had
$$\frac{1}{\log\lr{\leb C^{1/n}/D_\leb}}\leq  \lr{1+\frac{\ep}{16\sqrt{2}}}^{1/3},$$
and the latter is true since 
$$D_\leb\leq  \frac{\lr{1+\frac{\ep}{16\sqrt{2}}}^{1/3}}{\sqrt{\vol_g(M)}}.$$
We then conclude from \eqref{sup bound}  and \eqref{E: eqn 1} that
\[\E{ \sup_{x \in M} \phi_{\lambda}(x)} \leq \lr{8\sqrt{2}+\frac{\ep}{2}}\sqrt{\frac{n \log \leb}{\vol_g(M)}}\]
for all $\leb \geq \leb_0.$ Since $\phi_{\lambda}$ is symmetric, 
\begin{equation}
\E{\norm{\phi_{\lambda}}_\infty} \leq 2\, \E{\sup_{x \in M} \phi_{\lambda}(x)}\leq \lr{16\sqrt{2}+\ep}\sqrt{\frac{n \log \leb}{\vol_g(M)}}\label{E:Final}
\end{equation}
Taking the $\limsup$ as $\leb\gives \infty$ in \eqref{E:Final} and then $\ep\gives 0$ completes the proof.
\begin{remark}\label{R: proof of Eval aperiodic}
The proof of the result stated in Remark \ref{R: Eval aperiodic} for the aperiodic case follows from
using the diameter asymptotics in Remark \ref{R: diameter aperiodic} which give ${D_\leb =  \frac{1}{\sqrt{2\vol_g(M)}}+o(1)}$ instead of $D_\leb \leq  \frac{1}{\sqrt{\vol_g(M)}}+o(1)$.
\end{remark}

\begin{claim}[Proof of \eqref{E:claim}]\label{Claim}
As $a\gives 0,$
\begin{equation*}
\abs{\int_0^1 \lr{1-a\cdot \log x}^{1/2}dx -1}\leq \frac{a}{2}.
\end{equation*}
\end{claim}
\begin{proof}
Making the change of variables $u=\lr{1-a\log x}^{1/2}$ we get 
\begin{equation}\label{E:Integral II}
\int_0^1 \lr{1-a\cdot \log x}^{1/2}dx = \frac{2}{a}e^{\frac{1}{a}} \int_1^\infty u^2e^{-\frac{u^2}{a}}\,du.
\end{equation}
Observe that $-\frac{a}{2u}\frac{\partial}{\partial u}$ preserves $e^{-u^2/a}$ and integrate by parts in (\ref{E:Integral II}) to get
\begin{align*}
\int_0^1 \lr{1-a\cdot \log x}^{1/2}dx &=1+e^{\frac{1}{a}} \int_1^\infty e^{-\frac{u^2}{a}}\,du\\
&=1+e^{\frac{1}{a}}{\sqrt\frac{a}{2}}\int_{\sqrt{\frac{2}{a}}}^\infty e^{-\frac{t^2}{2}}\,dt
\end{align*}
Using the classical estimate
$$\int_x^{\infty} e^{-\frac{y^2}{2}}\,dy \leq \frac{e^{-\frac{x^2}{2}}}{x}$$
shows that, as $a\gives 0,$
\begin{equation*}
\abs{\int_0^1 \lr{1-a\cdot \log x}^{1/2}dx - 1}\leq \frac{a}{2},
\end{equation*}
as desired.
\end{proof}

%---------------------------------------------------------------------------------------------

%---------------------------------------------------------------------------
% BIBLIOGRAPHY
%--------------------------------------------------------------------------------

\end{document}